\documentclass[12pt]{amsart}

\usepackage{etex}
\usepackage{amsmath, amssymb}
\usepackage[frame,cmtip,arrow,matrix,line,graph,curve]{xy}
\usepackage{graphpap, color, paralist, pstricks}
\usepackage[mathscr]{eucal}
\usepackage[pdftex]{graphicx}
\usepackage[pdftex,colorlinks,backref=page,citecolor=blue]{hyperref}
\usepackage{tikz}
\usepackage{kotex}
\usepackage{array}
\newcolumntype{P}[1]{>{\centering\arraybackslash}p{#1}}
\newcolumntype{M}[1]{>{\centering\arraybackslash}m{#1}}

\newtheorem{theorem}{Theorem}[section]
\newtheorem{proposition}[theorem]{Proposition}
\newtheorem{corollary}[theorem]{Corollary}
\newtheorem{lemma}[theorem]{Lemma}

\theoremstyle{definition}

\newtheorem{example}[theorem]{Example}

\newtheorem{remark}[theorem]{Remark}

\newtheorem{notation}[theorem]{Notation}

\newcommand{\PP}{\mathbb{P}}

\newcommand{\CC}{\mathbb{C}}

\newcommand{\ZZ}{\mathbb{Z}}

\newcommand{\cO}{\mathcal{O} }

\newcommand{\cK}{\mathcal{K} }
\newcommand{\cM}{\mathcal{M} }
\newcommand{\cN}{\mathcal{N} }

\newcommand{\cU}{\mathcal{U} }

\newcommand{\rE}{\mathrm{E} }
\newcommand{\rIE}{\mathrm{IE} }
\newcommand{\rH}{\mathrm{H} }

\newcommand{\rP}{\mathrm{P} }
\newcommand{\rIP}{\mathrm{IP} }

\newcommand\bD{\mathbf{D}}
\newcommand\bE{\mathbf{E}}
\newcommand\bF{\mathbf{F}}

\newcommand\bH{\mathbf{H}}

\newcommand\bS{\mathbf{S}}

\def\Ext{\mathrm{Ext} }

\def\Gr{\mathrm{Gr} }

\def\lr{\rightarrow}

\newcommand{\ses}[3]{0\lr{#1}\lr{#2}\lr{#3}\lr 0}

\hypersetup{%
pdftitle={Double lines in quintic del Pezzo $4$-fold},%
pdfauthor={Kiryong Chung},%
pdfkeywords={moduli of sheaves, elementary modification, moduli space},%
citecolor=blue,%
linkcolor=blue,%
}

\begin{document}

\title{Double lines in the quintic del Pezzo fourfold}
%\date{\today}

\author{Kiryong Chung}
\address{Department of Mathematics Education, Kyungpook National University, 80 Daehakro, Bukgu, Daegu 41566, Korea}
\email{krchung@knu.ac.kr}


\begin{abstract}
Let $Y$ be the del Pezzo $4$-fold defined by the linear section $\textrm{Gr}(2,5)$ by $\mathbb{P}^7$. In this paper, we classify the type of normal bundles of lines in $Y$ and describe its parameter space. As a corollary, we obtain the desigularized model of the moduli space of stable maps in $Y$. Also we compute the intersection Poincar\'e polynomial of the stable maps space.
\end{abstract}
\maketitle

\section{Introduction}
\subsection{Motivation}

In previous series of papers \cite{CK11, CHK12, CM17}, the authors completely solved the comparison problem of different moduli spaces (i.e., the stable maps space, Hilbert scheme of curves and the stable sheaves space) of rational curves in a homogeneous variety $X$ when the degree of curves is $\leq 3$. As a result, we obtain the moduli theoretic birational model (in the sense of log minimal model program) and compute the cohomology group of the moduli spaces.
In this case, the convexity of $X$ provides the mild singularity of the moduli space of stable maps and thus one can set it as the initial point of comparison. But many of Fano varieties are not convex. As such a toy example is the minimal compactification of $\CC^3$: the quintic del Pezzo $3$-fold $W_5$ and Mukai variety $W_{22}$. In the case of $W_5$, our starting point of the comparison is the Hilbert scheme (which is isomorphic to the moduli space of stable sheaves). In \cite{Chu19}, we obtain the desingularized model of the moduli space of stable maps in $W_5$. In this paper, as well-known example of the minimal compactification of $\CC^4$, we study the rational curves in quintic del Pezzo $4$-fold $Y$ which is unique up to isomorphism. We deal with the first non-trivial case, that is, the degree two rational curves in $Y$. We obtain the desingularized model of stable maps space and thus its intersection cohomology group. Similar to the $3$-fold case (\cite{Chu19}), the crucial part is to classify types of the normal bundle of a line in $Y$. In general, the geometry of lines in Fano variety has been an important role for determining the geometric properties of Fano variety (\cite{PZ16, KPS18, PCS19}).
%For example, the automorphism group of the quintic del Pezzo threefold can be described by studying the secant and tangent variety of lines.
\subsection{Results}
Unless otherwise stated we define the quintic del Pezzo $4$-fold $Y$ by the linear section $\Gr(2,5)$ by $\{p_{12}-p_{03}=p_{13}-p_{24}=0\}$ where $\{p_{ij}\}$ are the Pl\"ucker coordinate of $\PP^9$ under the Pl\"ucker embedding $\Gr(2,5)\subset \PP^9$. It is known that the normal bundle $N_{L/Y}$ of a line $L$ in $Y$ is one of the following types (\cite[Lemma 1.6]{PZ16})
\[
N_{L/Y}\cong \cO_L(1)\oplus \cO_L^{\oplus 2}\; \mathrm{or} \; \cO_L(-1)\oplus \cO_L(1)^{\oplus 2}.
\]
Let us call the line of the first case (resp. the second case) by free line (resp. non-free line). Let $\bH_d(Y)$ be the Hilbert scheme of curves $C$ with Hilbert polynomial $\chi(\cO_C(m))=dm+1$ in $Y$.
Let us define the \emph{double line} $L^2$ as the non-split extension sheaf $F$ ($\cong\cO_{L^2}$)
\[
\ses{\cO_L(-1)}{F}{\cO_L},
\]
where $L$ is a line.
%Here $L^2$ is a non-reduced plane conic (i.e., the reduced support of $F$ is $\mathrm{red}\mathrm{Supp}(F)=L$ and the Hilbert polynomial of $F$ is $\chi(F(m))=2m+1$).
A double line $L^2$ in $Y$ supported on $L$ is classified by $\Ext^1(\cO_L, \cO_L(-1))\cong \rH^0(N_{L/Y}(-1))$. Hence the double line $L^2$ for a (resp. non-free) free line $L$ in $Y$ is unique (resp. isomorphic to $\PP^1$). The main result of this paper is the following.
\begin{theorem}\label{mainthm}
Let $\bD(Y)$ be the locus of double lines in $\bH_2(Y)$. Then $\bD(Y)$ is a $4$-dimensional smooth subvariety of $\bH_2(Y)$.
\end{theorem}
By combining the result of \cite{CHL18} and Theorem \ref{mainthm}, it turns out that non-free lines in $Y$ consist of lines meeting with the \emph{dual conic} $C_{v}^{\vee}$ at a point (Corollary \ref{mainprop}). Furthermore we obtain the designularized model (i.e., a subvariety of \emph{complete conics}) of the moduli space of stable maps of degree $2$ in $Y$, which enable to compute the intersection cohomology group of the moduli space. For detail description, see Corollary \ref{inter}.

\begin{notation}
\begin{itemize}
\item Let us denote by $\Gr(k,n)$ the Grassimannian variety parameterizing $k$-dimensional subspaces in a fixed vector space $V$ with $\dim V=n$.
\item We sometimes do not distinguish the moduli point $[x]\in \cM$ and the object $x$ parameterized by $[x]$ when no confusion can arise.
\item Let us shortly denote the projectivized linear subspace $\PP(e_i,\cdots, e_j)=\PP(\text{span}\{e_i,\cdots, e_j\})$ in $\PP(V_5)$ where $\{e_0,e_2,\cdots, e_4\}$ is the standard basis of the vector space $V_5$ of dimension $5$.
\end{itemize}
\end{notation}

\subsection*{Acknowledgements}
The author gratefully acknowledges the many helpful suggestions of Sang-Bum Yoo and Joonyeong Won during the preparation of the paper. 

%%%%%%% %%%%%%% %%%%%%% %%%%%%% %%%%%%% %%%%%%% %%%%%%% %%%%%%% %%%%%%% %%%%%%% %%%%%%% %%%%%%% %%%%%%% %%%%%%% %%%%%%% %%%%%%% %%%%%%% %%%%%%% %%%%%%% %%%%%%% %%%%%%% %%%%%%% %%%%%%% %%%%%%% %%%%%%% %%%%%%% %%%%%%% %%%%%%% %%%%%%% %%%%%%% %%%%%%% %%%%%%% %%%%%%% %%%%%%% %%%%%%% %%%%%%% %%%%%%% %%%%%%% %%%%%%% %%%%%%% 

\section{Preliminaries}\label{sec:preliminaries}
In this section we collect some facts about the quintic del Pezzo fourfold which are mostly taken from \cite{Pro94} and \cite{CHL18}.

One can define a Schburt variety relating with lines and planes in $\Gr(2,5)$ as follow. For fixed a flag $p\in \PP^1\subset \PP^2\subset \PP^3\subset \PP^{4}$, let
\begin{itemize}
\item $\sigma_{3,2}=\{\ell\,|\, p\in\ell\subset \PP^2\}$,
\item $\sigma_{3,1}=\{\ell\,|\, p\in \ell\subset \PP^3\}$,
\item $\sigma_{2,2}=\{\ell\,|\, \ell\subset \PP^2\}$.
\end{itemize}
Clearly, $\sigma_{3,2}$ is a line in $\Gr(2,5)$ and thus it is parameterized by $\Gr(1,3,5)$.
Also, we note that the planes in $\text{Gr}(2,5)$ with $\sigma_{3,1}$(resp. $\sigma_{2,2}$)-type is parameterized by the \emph{flag} variety $\text{Gr}(1, 4, 5)$ (resp. $\text{Gr}(3,5)$). The projection maps $v_1:\Gr(1,3,5)\to \Gr(1,5)$ and $v_2:\Gr(1,4,5)\to \Gr(1,5)$ are called by the \emph{vertex map}. In \cite{Pro94}, the Hilbert scheme of lines and planes in $Y$ are explicitly described. For a projective variety $X$ with fixed embedding in $\PP^N$, let $\bH_1(X)$ (resp. $\bF_2(X)$) be the Hilbert scheme of lines (resp. planes) in $X$
\begin{proposition}[\protect{\cite[Proposition 2.7]{Pro94}}]\label{linespace}
Let $i:\bH_1(Y)\subset \bH_1(\Gr(2,5))= \Gr(1,3,5)$ be the inclusion map and $v_1: \Gr(1,3,5)\lr \Gr(1,5)$ be the vertex map. Then the composition map $v_1\circ i: \bH_1(Y)\to \Gr(1,5)$ is a smooth blow-up along the smooth conic $C_{v}\subset \Gr(1,5)$.
\end{proposition}
Let us call $C_{v}$ by the \emph{vertex conic} in Proposition \ref{linespace}.
\begin{proposition}[\protect{\cite[Proposition 2.2]{Pro94}}]\label{planespace1}
The Hilbert scheme of planes in $Y$ is isomorphic to
\[
\bF_2(Y)\cong C_{v}\sqcup \{[S]\}.
\]
Here each point $t\in C_{v}(\cong\PP^1)$ parameterizes the $\sigma_{3,1}$-type planes $P_t$ such that the vertex of the plane $P_t$ is the point $\{t\}$ in $C_{v}$. Also the point $[S]$ parameterizes the $\sigma_{2,2}$-type plane $S$ in $Y$ determined by the linear spanning $\langle C_{v}\rangle\subset \Gr(1,5)\cong \PP^4$ of $C_{v}$.
\end{proposition}
Let $\{e_0,e_1,e_2,e_3, e_4\}$ be the standard coordinate vectors of the space $V_5(\cong \CC^5)$, which provides the original projective space $\PP(V_5)(=\PP^4)$. Let $\{p_{ij}\}_{0\leq i<j\leq 4}$ be the Pl\"ucker coordinates of $\PP^9$. Let $\PP^7=H_1\cap H_2$ be the linear subspace of $\PP^9$ defined by $p_{12}-p_{03}=p_{13}-p_{24}=0$. The vertex conic $C_{v}$ is given by (\cite[Lemma 6.3]{CHL18})
\[
C_{v}=\{[a_0:a_1:a_2:a_3:a_4]\mid a_0a_4+a_1^2=a_2=a_3=0\}\subset \PP(V_5).
\]
\begin{remark}\label{planeeq}
From the proof of \cite[Lemma 6.3]{CHL18}, we know that $\sigma_{3,1}$-type planes $P_t$ in $Y$ are $P_t=\PP(V_1\wedge V_4)$ where $V_1=\text{span}\{e_0+te_1-t^2e_4\}$ and $V_4=\text{span}\{e_0,e_1,e_2+te_3,e_4\}$. Also the unique plane $S$ in $Y$ is given by $S=\PP(\wedge^2 V_3)$ such that $V_3=\text{span}\{e_0,e_1,e_4\}$.
\end{remark}
The positional relation of planes in $Y$ are as follows.
\begin{proposition}[\protect{\cite[Proposition 2.2]{Pro94}}]\label{planespace2}
Let $P_t$ be a $\sigma_{3,1}$-type plane and $S$ be the unique $\sigma_{2,2}$-type plane in $Y$. Then
\begin{enumerate}
\item the intersection part $P_t\cap S$ is tangent line of the dual conic\footnote{That is, the curve is generated by the tangent lines of $C_{v}$.} $C_{v}^{\vee}$ in $Y$.
\item the intersection part $P_t\cap P_{t'}$ is a point in $S$ for any $t\neq t' \in C_{v}$.
\end{enumerate}
\end{proposition}
The lines in $Y$ have a stratification relating with the plane's type in $Y$.
\begin{proposition}[\protect{\cite[Corollary 3.7]{Pro94}}]\label{autooribit}
Let $L$ be a line in $Y$ and $R=\bigcup\limits_{t\in C_{v}} P_t$ be the union of planes in $Y$. Then there are five types of lines in $Y$ such that the automorphism group $\mathrm{Aut}(Y)$ of $Y$ transitively acts on each stratum.
\begin{enumerate}[(a)]
\item $L\nsubseteq R\cup S$.
\item $L\subset R$, $L\cap S=\{\mathrm{pt}.\}$ and $L\cap C_{v}^{\vee}=\emptyset$.
\item $L\subset R$, $L\cap S=L \cap C_{v}^{\vee}=\{\mathrm{pt}.\}$.
\item $L\subset S$ and $L$ is a tangent line of $C_{v}^{\vee}$.
\item $L\subset S$ and $L\cap C_{v}^{\vee}=\{p_1, p_2\}$ for $p_1\neq p_2$.
\end{enumerate}
\end{proposition}

In Section $6$ of \cite{CHL18}, the authors reproduce the results of Proposition \ref{linespace}, \ref{planespace1}, and \ref{planespace2} by specifying the linear subspace $\PP^7\subset \PP^9$.
\begin{example}
Let $P_{t_0}$ be the plane determined by the vertex $\PP(e_0)$ and the three dimensional space $\PP(e_0,e_1,e_2,e_4)$. The intersection point is $P_{t_0}\cap C_{v}^{\vee}=\PP(e_0\wedge e_1)$ which is the tangent line of $C_{v}$ at $\PP(e_0)$.
Furthermore, the example of lines in Proposition \ref{autooribit} are given in Table \ref{linesinY}.
%Let us shortly denote the projectivized subspace $\PP(e_i,\cdots, e_j)=\PP(\text{span}\{e_i,\cdots, e_j\})$ in $%\PP(V_5)$.
\begin{table}[!ht]
\begin{tabular}{|l|l|p{4cm}|p{4cm}|}
\hline
Type&Vertex&\;\;\;\;\;\;\;\;\;\;\; Plane\\
\hline
\hline
\;(a) &$\PP(e_2)$&$\;\;\;\;\;\;\PP(e_0,e_2,e_3)$\\
\hline
\;(b) &$\PP(e_0)$&$\;\;\;\;\;\;\PP(e_0,e_2,e_4)$\\
\hline
\;(c) &$\PP(e_0)$&$\;\;\;\;\;\;\PP(e_0,e_1,e_2)$\\
\hline
\;(d) &$\PP(e_0)$&$\;\;\;\;\;\;\PP(e_0,e_1,e_4)$\\
\hline
\;(e) &$\PP(e_1)$&$\;\;\;\;\;\;\PP(e_0,e_1,e_4)$\\
\hline
\end{tabular}
\caption{Example of lines in $Y$}
\label{linesinY}
\end{table}
\end{example}
Let $\bH_2(Y)$ be the Hilbert scheme of conics in $Y$. For a general conic $C$ in $Y$, it determines linear spanning in two meanings: the linear space $\PP^2$ containing $C$ in $\PP(\wedge^2V_5)=\PP^9$ and the linear space $\PP^3$ containing two skew lines in $\PP(V_5)=\PP^5$. Motivated this observation, we have a birational model of $\bH_2(Y)$ as follow. Let $\cU$ be the universal subbundle on $\Gr(4,5)$ and
\[\cK:= \mathrm{ker}\{\wedge^2\cU\subset \wedge^2 \cO^{\oplus 5}\to \cO^{\oplus 2} \}\] be the kernel of the composition map where the arrow is given by $\{p_{12}-p_{03},p_{13}-p_{24}\}$. Let $\bS(Y):=\Gr(3, \cK)$ be the relative Grassmannian over $\Gr(4,5)$.
\begin{proposition}[\protect{\cite[Proposition 6.7 and Remark 6.8]{CHL18}}]\label{bidiagconic}
Under above definition and notation, $\bH_2(Y)$ is obtained from $\bS(Y)$ by a blow-down followed by a blow-up
\begin{equation*}
\xymatrix{
&\widetilde{\bS}(Y)\ar[rd]\ar[ld]&\\
\bS(Y)\ar@{-->}[rr]^{\Psi}&& \bH_2(Y),
}
\end{equation*}
where
\begin{enumerate}
\item the blow-up center in $\bS(Y)$ (resp. $\bH_2(Y)$) is a disjoint union $\PP^1\sqcup \PP^1$ (resp. $\PP^5$)  of $\PP^1$'s and
\item the space $\widetilde{\bS}(Y)$ is a relative conics space over $\Gr(4,5)$ such the fiber over $\Gr(4,5)$ is the Hilbert scheme $\bH_2(\Gr(2,4)\cap H_1\cap H_2)$ of conics in the quadric surface $\Gr(2,4)\cap H_1\cap H_2$.
\end{enumerate}
\end{proposition}
In special $\bH_2(Y)$ is an irreducible and smooth variety of dimension $7$. 
\begin{remark}\label{corrbi}
The relative Grassimannian $\bS(Y)=\Gr(3, \cK)$ in Proposition \ref{bidiagconic} can be regarded as the incident variety
\begin{equation}\label{incident}
\bS(Y)=\{(U_3, V_4) \mid U_3\subset \cK_{[V_4]}\}\subset \Gr(3, \wedge^2 V_5)\times \Gr(4, V_5),
\end{equation}
where $\cK_{[V_4]}=\text{ker}\{\wedge^2V_4\subset \wedge^2V_5\stackrel{(p_{12}-p_{03})\oplus (p_{13}-p_{24})}{\lr}\CC\oplus \CC\}$.
Also the birational correspondence $\Psi:\bS(Y)\dashrightarrow  \bH_2(Y)$ is $\Psi([(U_3, V_4)])=\PP(U_3)\cap \Gr(2,V_4)$. Note that the map $\Psi$ is not defined at the two distinct points $[P_t]$ and $[S]$ over a linear subspace $\PP^1(\cong C_v)$ in $\Gr(4,5)$ (\cite[Remark 6.8]{CHL18}).
\end{remark}
%%%%%%% %%%%%%% %%%%%%% %%%%%%% %%%%%%% %%%%%%% %%%%%%% %%%%%%% %%%%%%% %%%%%%% 
%%%%%%% %%%%%%% %%%%%%% %%%%%%% %%%%%%% %%%%%%% %%%%%%% %%%%%%% %%%%%%% %%%%%%% %%%%%%% %%%%%%% %%%%%%% %%%%%%% %%%%%%% %%%%%%% %%%%%%% %%%%%%% %%%%%%% %%%%%%% 
%%%%%%% %%%%%%% %%%%%%% %%%%%%% %%%%%%% %%%%%%% %%%%%%% %%%%%%% %%%%%%% %%%%%%% 
\section{Results}
In this section we prove Theorem \ref{mainthm}. As corollaries, we have a description of the locus of non-free lines in $Y$ (Corollary \ref{mainprop}). Also we obtain the desinguarized model of stable maps space in $Y$ and thus its intersection cohomology (Corollary \ref{inter}).
\subsection{Proof of Theorem \ref{mainthm}} 
Firstly, we describe the closure of the birational inverse $\Psi^{-1}$ of the double line in $\bH_2(Y)$ in Proposition \ref{bidiagconic}. Then we find explicitly the strict transform of the closure along the blow-up/down maps in Proposition \ref{bidiagconic}.

Let $\bar{\bD}(Y)$ be the locus of the pairs $(U_3,V_4)$ in $\bS(Y)$ such that the restriction $q_G|_{\PP(U_3)}$ to $\PP(U_3)$ of the quadric form $q_G$ associated to $\Gr(2,V_4)$ is rank $\leq 1$. %Geometrically, $\PP(U_3)$ is tangential to singular quadric surface $G(2,V_4)\cap H_1\cap H_2$.
Let \[p=p_2\circ i:\bar{\bD}(Y)\lr \Gr(4,5)\] be the composition of the second projection map $p_2:\bS(Y)\lr \Gr(4,5)$ in equation \eqref{incident} and the inclusion map $i: \bar{\bD}(Y)\subset \bS(Y)$.
\begin{lemma}
Under above definition and notations, the image $p(\bar{\bD}(Y)):=Q_3$ is an irreducible quadric $3$-fold in $\Gr(4,5)$ with the homogenous coordinates $x_0, x_1, x_2, x_3, x_4$ such that $Q_3$ is defined by $x_1^2+4x_0x_2=0$.
%and $\text{Sing}(Q_3)(\cong \PP^1)$ is defined by $I_{\mathrm{Sing}(Q_3)}=\langle x_0, x_1, x_2\rangle$. 
\end{lemma}
\begin{proof}
For the chart $x_3\neq 0$, let
\[[V_4]:=\left( \begin{matrix}
1&0&0&a&0\\
0&1&0&b&0\\
0&0&1&c&0\\
0&0&0&d&1\\
\end{matrix}\right)\]
be an affine chart of $\Gr(4,5)$ with $a=x_0/x_3, b=x_1/x_3, c=x_4/x_3, d=x_2/x_3$ and $V_4=\text{span}\{e_0+ae_3, e_1+be_3, e_2+ce_3, e_4+de_3\}$. Then for  an affine chart of $\Gr(2,V_4)$
\[[V_2]=\left( \begin{matrix}
1&0&t_1&t_3\\
0&1&t_2&t_4\\
\end{matrix}\right),\]
the affine chart of $\Gr(2,V_4)$ in $\Gr(2,5)$ is
\[
[V_2][V_4]=\left( \begin{matrix}
1&0&t_1&a+ct_1+dt_3&t_3\\
0&1&t_2&b+ct_2+dt_4&t_4
\end{matrix}\right).
\]
After eliminating the variables $\{t_1,t_2,t_3,t_4\}$ by the computer program (\cite{M2}), we have a defining equation of $\Gr(2,V_4)\cap H_1\cap H_2$
\begin{equation}\label{defofquad1}
\begin{split}
\langle bcp_{01}^2+c^2p_{01}p_{02}&+cdp_{01}p_{04}-ap_{01}^2+bp_{01}p_{04}+cp_{02}p_{04}+dp_{04}^2+dp_{01}p_{14}-p_{02}p_{14},\\
\;h_1,\;h_2,\;h_3,&\;h_4,\;h_5,\;h_6 \rangle
\end{split}
\end{equation}

where
\[ 
\begin{split}
h_1&=p_{03}-p_{12}, \;h_2=p_{12}-bp_{01}-cp_{02}-dp_{04}, \;h_3=p_{13}-p_{24},\\
 h_4&=p_{23}+ap_{02}+bp_{12}-dp_{24}, \;h_5=p_{24}+ap_{01}-cp_{12}-dp_{14},\; h_6=p_{34}-ap_{04}-bp_{14}-cp_{24}
 \end{split}
 \]
%h_1= a*p_{12}+c*p_{23}+b*p_{24}-d*p_{34},
in $\PP^9\times \CC_{(a,b,c,d)}$.

For the chart $x_4\neq 0$, let $a=x_0/x_4, b=x_1/x_4, u=x_3/x_4, d=x_2/x_4$. By doing the same calculation as before, we obtain the local equation of $\Gr(2,V_4)\cap H_1\cap H_2$ as follows:
\begin{equation}\label{defofquad2}
\begin{split}
\langle ap_{04}^2+ap_{01}p_{14}+bp_{04}p_{14}&-dp_{14}^2-p_{01}p_{34}-au p_{04}p_{14}-bu p_{14}^2-up_{04}p_{34}+u^2p_{14}p_{34},\\
\;k_1,\;k_2,\;k_3,&\;k_4,\;k_5,\;k_6 \rangle
\end{split}
\end{equation}
where
\[ 
\begin{split}
k_1=&p_{02}+bp_{01}+dp_{04}-up_{12}, \;k_2=p_{03}-p_{12}, \;k_3=p_{12}-ap_{01}+dp_{14}-up_{24},\\
k_4=&p_{13}-p_{24},\;k_5=p_{23}+ap_{12}+bp_{24}-dp_{34},\;k_6=p_{24}+ap_{04}+bp_{14}-u p_{34}
 \end{split}
 \]
in $\PP^9\times \CC_{(a,b,u,d)}$.

Since the restricted form $q_G$ associated to $\Gr(2,V_4)$ is of $\text{rank}(q_G|_{\PP(U_3)})\leq 1$, the quadratics of the defining equations \eqref{defofquad1} and \eqref{defofquad2} is of rank $\leq 3$. Therefore the defining equation of the image $p(\bar{\bD}(Y))$ is given by
\[
\langle b^2+4ad\rangle
\]
in both cases.
\end{proof}
Obviously, the singular locus of $\text{Sing}(Q_3)(\cong \PP^1)$ is defined by $I_{\mathrm{Sing}(Q_3)}=\langle x_0, x_1, x_2\rangle$.
\begin{proof}[Proof of Theorem \ref{mainthm}]
\textbf{Step 1}. For each $[V_4]\in Q_3\setminus \text{Sing}(Q_3)$, the quadric surface $\Gr(2,V_4)\cap H_1\cap H_2$ is of rank $3$. Hence the fiber $p^{-1}([V_4])$ is isomorphic to $\PP^1$ which parameterizes tangent planes (i.e., lines) of the quadric cone $\Gr(2,V_4)\cap H_1\cap H_2$. If $[V_4]\in\text{Sing}(Q_3)$, the singular quadric surface $\Gr(2,V_4)\cap H_1\cap H_2$ is the union of the plane $P_{t}$ and $S$. In fact, for the affine chart $x_3\neq 0$ (similarly, $x_4\neq 0$), it is defined by 
the union of $\sigma_{3,1}$-type planes:
\[
\langle 
c^2p_{01}+cp_{04}-p_{14}, p_{23},cp_{24}-p_{34},cp_{02}-p_{12},-cp_{12}+p_{24}, p_{13}-p_{24},p_{03}-p_{12}
\rangle
\]
and the $\sigma_{2,2}$-plane $S$:
\[
\langle 
p_{02}, p_{23},cp_{24}-p_{34},cp_{02}-p_{12},-cp_{12}+p_{24}, p_{13}-p_{24},p_{03}-p_{12}
\rangle
\]
which matches with Remark \ref{planeeq} (by letting $c=t$). Hence the fiber $p^{-1}([V_4])$ is isomorphic to $\PP^1$ which parameterizes planes containing the intersection line $P_{t}\cap S$. After all, $\bar{D}(Y)$ is a $\PP^1$-fiberation over $Q_3$.

\textbf{Step 2}. Note that the birational map $\Psi$ in Proposition \ref{bidiagconic} is not defined for the two points: $\{[P_{t}],[S]\}$ over $\text{Sing}(Q_3)(\cong \PP^1)$. Hence the blow-up center of $\eta:\widetilde{\bS}(Y)\lr \bS(Y)$ is contained in $\bar{\bD}(Y)$ and thus the strict transform of $\bar{\bD}(Y)$ by the blow-up map $\eta$ is nothing but the blow-up $\widetilde{\bD}(Y)$ of $\bar{\bD}(Y)$ along the center $\PP^1\sqcup \PP^1$. Since the blow-center of $\bar{\bD}(Y)$ is of $\ZZ_2$-quotient singularity, one can easily check that $\widetilde{\bD}(Y)$ is smooth and the exceptional divisor $\bE$ in $\widetilde{\bD}(Y)$ is a $\PP(1,2,2)(\cong \PP^2)$-bundle over $\PP^1\sqcup \PP^1$. Each fiber $\PP^2$ parameterizes the double line in the plane because any flat family in $\bar{\bD}(Y)$ is obviously supported on lines by its construction.

\textbf{Step 3}. The restriction to each fiber $\PP^1$ of the normal bundle $\cN_{\bE/\widetilde{\bD}(Y)}$ of the exceptional divisor $\bE$ is $\cN_{\bE/\widetilde{\bD}(Y)}|_{\PP^1}\cong \cO_{\PP^1}(-1)$, the image $\bD(Y)$ of the restriction to $\widetilde{\bD}(Y)$ of the blow-down map $\widetilde{\bS}(Y)\lr \bH_2(Y)$ is smooth by the Fujiki-Nakano criterion (\cite{FN71}). So we finish the proof.
\end{proof}

\subsection{Non-free lines in $Y$ and the intersection cohomology of stable maps}\label{interstable}
\begin{corollary}\label{mainprop}\label{maincor}
Let $Z$ be the locus of non-free lines in the Hilbert scheme $\bH_1(Y)$ of lines in $Y$. Then $Z$ is isomorphic to a $\PP^1$-fiberation over the vertex conic $C_{v}(\cong \PP^1)$.
\end{corollary}
\begin{proof}
Geometrically, non-free lines are lines in $Y$ meeting with the dual conic $C_{v}^{\vee}$ at a point uniquely. Since the automorphism of $Y$ transitively acts on each stratum of Proposition \ref{autooribit}, it is enough to check each case in Table~\ref{linesinY}. For the case (d) and (e), $L$ is non-free line if and only if $L$ is a tangent line of $C_{v}^{\vee}$ by \cite[Proposition 6.6]{CHL18}. Thus the lines of the case (d) are only non-free. For the case (c), the line $L$ is defined by $p_{03}=p_{04}=p_{12}=p_{13}=p_{14}=p_{23}=p_{24}=p_{34}=0$. Thus for the affine chart $x_3\neq 0$, it lies on irreducible quadric cones defined by $dp_{04}^2+(dp_{01}-p_{02})p_{14}=h_1=p_{12}-dp_{04}=h_3=p_{23}-dp_{24}=p_{24}-dp_{14}=p_{34}=0$ for $d\neq 0$. Hence there exists one parameter family of double lines supported on $L$. That is, $L$ is non-free. For other cases (a) and (b), we know that each line is free by a similar computation.
\end{proof}
Let $C$ be a projective connected reduced curve. A map $f: C \to Y$ is considered \emph{stable} if $C$ has at worst nodal singularities and $|\mathrm{Aut}(f)|<\infty$. Let $\cM(Y,d)$ be the moduli space of isomorphism classes of stable
maps $f:C\to Y$ with genus $g(C)=0$ and $\mathrm{deg} (f^*\cO_{Y}(1))=d$. The moduli space $\cM(Y,d)$ might be singular and reducible depending on the geometric property (for example, convexity) of $Y$.
\begin{remark}\label{extob}
Let $f:C\lr L\subset Y$ be a stable map of the degree $\deg(f)=2$ such that $L$ is non-free. From the tangent bundle sequence of $L\subset Y$, $\rH^1(f^*T_Y)\cong \rH^1(f^*N_{L/Y})\cong\rH^1(f_*\cO_C\otimes N_{L/Y})\cong\CC$. That is, $Y$ is not convex and thus $\cM(Y,2)$ is not smooth stack (\cite{FP97}).
\end{remark}
Let $X$ be a quasi-projective variety. For the (resp. intersection) Hodge-Deligne polynomial $\rE_c(X)(u,v)$ (resp. $\rIE_c(X)(u,v)$) for compactly supported (resp. intersection) cohomology of $X$, let
\[ \rP(X)=\rE_c(X)(-t,-t)\; (\mathrm{resp.}\;\rIP(X)=\rIE_c(X)(-t,-t))\]
be the \emph{virtual} (resp. intersection) Poincar\'e polynomial of $X$. A map $\pi:X\lr Y$ is \emph{small} if for a locally closed stratification of $Y=\bigsqcup_i Y_i$ such that the restriction map $\pi|_{\pi^{-1}(Y_i)}:\pi^{-1}(Y_i)\lr Y_i$ is \emph{etale} locally trivial, the inequality \[\dim \pi^{-1}(y)< \frac{1}{2}\mathrm{codim}_Y(Y_i)\] holds for each closed point $y\in Y_i$ except a dense open stratum of $Y$. Let $\pi: X\lr Y$ be a small map such that $X$ has at most finite group quotient singularities. Then $\rP(X)=\rIP(Y)$ (\cite[Definition 6.6.1 and Theorem 6.6.3]{Max18}).

\begin{corollary}\label{inter}
The intersection cohomology of the moduli space $\cM(Y,2)$ is given by
\[\rIP(\cM(Y,2))= 1+4t^{2}+10t^{4}+15t^{6}+15t^{8}+10t^{10}+4t^{12}+t^{14}.\]
\end{corollary}
\begin{proof}
By the same method of the proof of Theorem 1.2 in \cite{Chu19}, one can show that the blow-up $\widetilde{\bH}_2(Y)$ of $\bH_2(Y)$ along $\bD(Y)$ is smooth one and thus we have a birational morphism
\[
\pi: \bH_2(Y)\lr \cM(Y,2)
\]
such that the exceptional divisor (i.e., $\PP^2$-bundle over $\bD(Y)$) contracts to a $\PP^2$-bundle over $\bH_1(Y)$. From Corollary \ref{maincor}, the map $\pi$ is a small map and thus $\rP(\widetilde{\bH}_2(Y))=\rIP(\cM(Y,2))$. By the equality
\[
\rP(\widetilde{\bH}_2(Y))=\rP(\bH_2(Y))+(\rP(\PP^2)-1)\cdot \rP(\bD(Y)),
\]
Proposition \ref{linespace} and Corollary \ref{maincor}, we obtain the result.
\end{proof}

%\bibliographystyle{alpha}
%\bibliography{Library}

\end{document}